\documentclass{amsart}
\usepackage{amsfonts, amsmath, amssymb, amsthm, mathtools, extarrows, hyperref,tikz-cd}
\usepackage{cite} 
\usepackage[small,nohug,heads=LaTeX]{diagrams}
\diagramstyle[labelstyle=\scriptstyle]
\newtheorem{theorem}{Theorem}[section]
\newtheorem{lemma}[theorem]{Lemma}
\newtheorem{proposition}[theorem]{Proposition}
\newtheorem{corollary}[theorem]{Corollary}
\theoremstyle{definition} 
 
\newtheorem{definition}[theorem]{Definition}
\newtheorem{example}[theorem]{Example}

\newcommand{\floor}{\lfloor \frac{P-2}{2}\rfloor}

\author{Gioia Failla, Zachary Flores, Chris Peterson}
\address{University of Reggio Calabria, Department DIIES
Via Graziella, Feo di Vito, Reggio Calabria}
\email{gioia.failla@unirc.it} 
\address{Department of Mathematics, Colorado State University,  Fort Collins, CO 80523, USA}
\email{flores@math.colostate.edu}
\address{Department of Mathematics, Colorado State University, Fort Collins, CO 80523, USA}
\email{peterson@math.colostate.edu}
\title{On the Weak Lefschetz Property for Vector Bundles on $\mathbb P^2$}
\begin{document}           

\begin{abstract} 
Let $R=\mathbb K[x,y,z]$ be a standard graded polynomial ring where $\mathbb K$ is an algebraically closed field of characteristic zero. Let $M = \oplus_j M_j$ be a finite length graded $R$-module.  We say that $M$ has the Weak Lefschetz Property if there is a homogeneous element $L$ of degree one in $R$ such that the multiplication map $\times L : M_j \rightarrow M_{j+1}$ has maximal rank for every $j$. The main result of this paper is to show that if $\mathcal E$ is a locally free sheaf of rank 2 on $\mathbb P^2$ then the first cohomology module of $\mathcal E$, $H^1_*(\mathbb P^2, \mathcal E)$, has the Weak Lefschetz Property.
 
\end{abstract}
\maketitle


\section{Introduction} 

Let $\mathbb K$ be an algebraically closed field of characteristic zero and let $R=\mathbb K[x,y,z]$ be the polynomial ring with standard grading $R=\oplus_{j\geq 0} R_j$.   Let $M = \oplus_{j\in \mathbb Z} M_j$ be a graded $R$-module.
We say that $M$ is of {\it finite length} if all but a finite number of the $M_j$ are equal to $0$ (with each $M_j$ a finite dimensional $\mathbb K$-vector space).  We say that $M$ has the {\it Weak Lefschetz Property} if there is a linear form $L \in R_1$ such that the $\mathbb K$-linear map $\times L:  M_j\rightarrow M_{j+1}$ has maximal rank for all $j$.  
Stanley and others showed how the Weak Lefschetz Property, a property that is geometric/algebraic in nature, ties in with several interesting problems of a combinatorial nature \cite{cook2011weak, li2010monomial, stanley1980number, stanley1980weyl}. In particular, Stanley utilized the property to complete the proof of McMullen's conjecture on the characterization of $f$-vectors of simplicial polytopes. In honor of the influential works of Stanley, the Weak Lefschetz Property is also referred to as the {\it Weak Stanley Property} in the literature. There has been a rich body of research establishing the existence or non-existence of the Weak Lefschetz Property for various types of Artinian algebras, in particular for Artinian Gorenstein algebras \cite{boij1999components, harima1995characterization, harima2003weak, iarrobino1994associated, watanabe1998note}  and other Artinian algebras with special structure \cite{mezzetti2013laplace, migliore2011monomial, zanello2006non}. Within this rapidly growing body of research involving the Weak Lefschetz Property, we found the following survey type works to be very helpful \cite{harima2013lefschetz, migliore2013survey}.

There were two papers that played a major role in inspiring us to utilize the vector bundle based techniques that ultimately led to a proof of our main result. The first was the paper by Harima et al \cite{harima2003weak} which made use of the Grauert-M\"ulich theorem to gain further insight into the Weak Lefschetz Property of a height three Artinian complete intersection.  The Grauert-M\"ulich theorem enabled them to pinpoint the generic splitting type of a stable, normalized, rank two vector bundle on $\mathbb P^2$ which enabled precise homological conclusions to be made. The second influential work for us was the paper by Brenner and Said \cite{brenner2007syzygy} which made further use of the Grauert-M\"ulich theorem for higher rank bundles on $\mathbb P^2$ and solidified the connection between the generic splitting type of a bundle and the Weak Lefschetz Property. 

It is very natural to study height three Artinian complete intersections via the Koszul complex. First of all, the Koszul complex is exact for complete intersections. Second, by sheafifying the Koszul complex, one can identify the first cohomology module of an associated rank two locally free sheaf as the Artinian module $R/(F_1, F_2, F_3)$ where $F_1,F_2,F_3$ is a regular sequence of homogeneous polynomials in $R$ defining the complete intersection. A natural generalization can be obtained via the Buchsbaum-Rim complex associated to an $R$-graded map $\phi:F\rightarrow G$ where $F$ is a free $R$-module of rank $n+2$ and $G$ is a free $R$-module of rank $n$.  In particular, if the cokernel of $\phi$ is of codimension three, which over $R=\mathbb K[x,y,z]$ corresponds to the cokernel being a module of finite length, then the Buchsbaum-Rim complex is exact. By sheafifying this complex we can again identify the first cohomology module of an associated rank two locally free sheaf, $\mathcal E$, as the cokernel of $\phi$. As in the papers \cite{brenner2007syzygy, harima2003weak}, it is crucial to understand the generic splitting type of $\mathcal E$ and its relationship to the multiplication between consecutive graded components of the cokernel of $\phi$ induced by a general linear form.

The paper is broken into four short sections. In section two of this paper we provide background meant to clarify the connection between the Buchsbaum-Rim complex for a certain class of finite length $\mathbb K[x,y,z]$-modules and rank 2 vector bundles on $\mathbb P^2$. The third section contains the statement and proofs of the main results of the paper. In particular, we show that the first cohomology module of any rank 2 vector bundle on $\mathbb P^2$ satisfies the Weak Lefschetz Property. The final section consists of examples, some potential paths for future research, and concluding remarks.

\section{The Buchsbaum-Rim complex}

Let $R=\mathbb K[x,y,z]$, let $\mathbf F=\oplus_{i=1}^{n+2} R(-a_i)$, let $\mathbf G=\oplus_{i=1}^{n} R(-b_i)$, let $a=a_1+\dots+a_{n+2}$, and let $b=b_1+\dots + b_n$. Given a graded degree zero map $\phi:\mathbf F \rightarrow \mathbf G$ we have a kernel $\mathbf E$, cokernel $\mathbf M$ and an exact sequence

\begin{equation}\label{SESModule}
0\rightarrow \mathbf E\rightarrow \mathbf F\rightarrow \mathbf G \rightarrow \mathbf M \rightarrow 0.
\end{equation}

In addition,  we have a Buchsbaum-Rim complex associated to $\phi:\mathbf F \rightarrow \mathbf G$ \cite{buchsbaum1964generalized,buchsbaum1964generalizedb}. If the cokernel of $\phi$ has the ``expected codimension'', which in this case corresponds to requiring that $\mathbf M$ has finite length, then the Buchsbaum-Rim complex is exact and has the form: 

\begin{equation}\label{BRModule}
0\rightarrow \mathbf G^{\vee}(b-a) \rightarrow \mathbf F^{\vee}(b-a) \rightarrow \mathbf F \rightarrow \mathbf G\rightarrow \mathbf M \rightarrow 0
\end{equation}

  This complex is one of a much larger family of complexes associated to sufficiently general maps between $R$-modules. These complexes are exact if a certain genericity condition is met and they can be derived by considering ``strands'' of a particular Koszul complex (see \cite{eisenbud2013commutative} Appendix A2.6 for details). 

If we sheafify (\ref{BRModule}) then we get an exact sequence of locally free sheaves 

\begin{equation}\label{BRSheaf}
0\rightarrow \mathcal G^{\vee}(b-a) \rightarrow \mathcal F^{\vee}(b-a) \rightarrow \mathcal F \rightarrow \mathcal G\rightarrow 0
\end{equation}
 which can be decomposed into the two short exact sequences 
 \begin{equation}\label{SES1}
 0\rightarrow \mathcal G^{\vee}(b-a) \rightarrow \mathcal F^{\vee}(b-a) \rightarrow \mathcal E\rightarrow 0
 \end{equation}
 \begin{equation}\label{SES2}
 0\rightarrow \mathcal E \rightarrow \mathcal F \rightarrow \mathcal G \rightarrow 0.
 \end{equation}
 Note that (\ref{SES2}) is also the sheafification of (\ref{SESModule}). The locally free sheaf $\mathcal E$ has rank two and is an example of a (first) Buchsbaum-Rim sheaf. The apparent symmetry of the Buchsbaum-Rim complex is closely related to the fact that a rank 2 locally free sheaf is self dual (up to a twist by a line bundle). In general, the structure found in the Buchsbaum-Rim complex is reflected in properties of $\mathcal E$, in properties of its sections, and in properties of its cohomology modules \cite{kreuzer2000determinantal, migliore1999buchsbaum}. In particular, the rigidity of the Buchsbaum-Rim complex, when it is exact, suggests that properties of the objects involved reduce to combinatorial considerations of the $a_i$ and $b_i$ involved in the definitions of $\mathbf F$ and $\mathbf G$. In the next section, we will see that this is indeed the case. Let $H^0_*(\mathbb P^2, \mathcal E)$ denote the module $\oplus_{i\in \mathbb Z}H^0(\mathbb P^2, \mathcal E(i))$. If we apply the global section functor to the short exact sequence $$0\rightarrow \mathcal E \rightarrow \mathcal F \rightarrow \mathcal G \rightarrow 0$$ we obtain the long exact sequence 
  \begin{equation}\label{LESCohomology}
  0\rightarrow H^0_*(\mathbb P^2, \mathcal E) \rightarrow H^0_*(\mathbb P^2, \mathcal F) \rightarrow H^0_*(\mathbb P^2, \mathcal G) \rightarrow H^1_*(\mathbb P^2, \mathcal E) \rightarrow  H^1_*(\mathbb P^2, \mathcal F) \rightarrow \dots.
  \end{equation}
  Note that $H^1_*(\mathbb P^2, \mathcal F)=0$ since $\mathcal F=\mathcal O_{\mathbb P^2}(-a_i)$ and that (\ref{LESCohomology}) is actually a recovery of (\ref{SESModule}). In particular, we have $$H^1_*(\mathbb P^2, \mathcal E)=\mathbf M.$$
In general, finite length $R=\mathbb K[x,y,z]$-modules that can be expressed as cokernels of maps of the form $\phi:\oplus_{i=1}^{n+2} R(-a_i) \rightarrow\oplus_{i=1}^{n} R(-b_i)$ correspond to finite length modules of the form $H^1_*(\mathbb P^2, \mathcal E)$ where $\mathcal E$ is a rank 2 locally free sheaf on $\mathbb P^2$.

\section{Main Results}
In this section, we collect the key definitions and theorems that form the heart of the paper. Many of the needed tools can be found in the book by Hartshorne on {\it Algebraic Geometry} \cite{hartshorne2013algebraic} and in the book by Okonek, Schneider, and Spindler on {\it Vector Bundles on Complex Projective Spaces} \cite{okonek1980vector}.

\begin{definition} Let $\mathcal E$ be a torsion free sheaf on $\mathbb P^n$. Let $c_1(\mathcal E)$ denote the first Chern class of $\mathcal E$ and let $rank(\mathcal E)$ denote its rank.
\begin{itemize}
\item[1)] The {\it slope} of $\mathcal E$ is the rational number $\mu(\mathcal E)=c_1(\mathcal E)/rank(\mathcal E)$
\item[2)]  $\mathcal E$ is said to be {\it stable} if for any non-zero subsheaf $\mathcal F\subset \mathcal E$ the slopes satisfy $\mu(\mathcal F)<\mu(\mathcal E)$
\item[3)] $\mathcal E$ is said to be {\it semistable} if for any non-zero subsheaf $\mathcal F\subset \mathcal E$ the slopes satisfy $\mu(\mathcal F)\leq\mu(\mathcal E)$
\item[4)] $\mathcal E$ is {\it unstable} if it is not semistable.
\end{itemize}
\end{definition}

In various contexts, the definition of stability given above is sometimes referred to by other names including slope stability, $\mu$-stability, Mumford stability, or Mumford-Takemoto stability. 

Let $\mathcal E$ be a vector bundle on $\mathbb P^n$ and let $r$ denote the rank of $\mathcal E$. We say that $\mathcal E$ is a {\it normalized} bundle if $c_1(\mathcal E)\in \{-r+1,\dots, 0\}$. In general, there exists a unique $a\in\mathbb Z$ such that $\mathcal E(a)$ is a normalized bundle. In particular, if $\mathcal E$ is a normalized rank 2 vector bundle, then $c_1(\mathcal E)\in \{-1,0\}$. The following lemma is a quick application of the definition of stability, see \cite{okonek1980vector} (Chapter II for a more detailed discussion of stability and  Lemma 1.2.5 on Pg 166-167 for the statement and proof of the lemma).

\begin{lemma}
Let $\mathcal E$ be an normalized rank 2 vector bundle on $\mathbb P^n$.
\begin{itemize}
\item[1)] $\mathcal E$ is stable if and only if $H^0(\mathbb P^n, \mathcal E)=0$.
\item[2)]  If $c_1(\mathcal E)=-1$ then $\mathcal E$ is semistable if and only if $\mathcal E$ is stable
\item[3)]  If $c_1(\mathcal E)=0$ then $\mathcal E$ is semistable if and only if $H^0(\mathbb P^n, \mathcal E(-1))=0$.
\end{itemize}
\end{lemma}

The following is the Grauert-M\"ulich Theorem for rank 2 bundles on $\mathbb P^n$. For a more detailed discussion of the Grauert-M\"ulich theorem and its role in the classification of vector bundles, see \cite{grauert1975vektorbundel} for the original result or see \cite{okonek1980vector} (Chapter II, section 2 for a general discussion of the splitting behavior of vector bundle and Corollary 2 on Pg 206 for the specifics of the Grauert-M\"ulich theorem).

\begin{proposition} \label{GM}
Let $\mathcal E$ be a semistable, normalized, rank 2 vector bundle on $\mathbb P^n$. Let $L$ be a general line.
\begin{itemize}
\item[1)]  If $c_1(\mathcal E)=0$ then the restriction to $L$ splits as $\mathcal E|_L \cong \mathcal O_{\mathbb P^1}\oplus \mathcal O_{\mathbb P^1}$.
\item[2)] If $c_1(\mathcal E)=-1$ then the restriction to $L$ splits as $\mathcal E|_L \cong \mathcal O_{\mathbb P^1}(-1)\oplus \mathcal O_{\mathbb P^1}$.
\end{itemize}
\end{proposition}

\begin{definition}
If $\mathcal E$ is an unstable, normalized, rank 2 vector bundle on $\mathbb P^n$ then the largest $a$ such that $H^0(\mathbb P^n, \mathcal E(-a))\neq 0$ is called the {\it index of instability} of $\mathcal E$. 
\end{definition}

From the above lemma, if $\mathcal E$ is an unstable, normalized, rank 2 vector bundle on $\mathbb P^n$ and $c_1(\mathcal E)=0$ then its index of instability is greater than zero. Similarly, if $c_1(\mathcal E)=-1$ then its index of instability is at least zero. If $\mathcal E$ is a vector bundle on $\mathbb P^2$, we can make a stronger statement:

\begin{proposition} \label{bigprop}
Let $\mathcal E$ be an unstable, normalized, rank 2 vector bundle on $\mathbb P^2$. Let $k$ be the index of instability of $\mathcal E$. Let $L$ be a general line in $\mathbb P^2$. Then
\begin{itemize}
\item[1)] Every nonzero section $s\in H^0(\mathbb P^2, \mathcal E(-k))$ is regular.
\item[2)] If $c_1(\mathcal E)=0$ then $k>0$ and $\mathcal E|_L=\mathcal O_{\mathbb P^1}(-k)\oplus \mathcal O_{\mathbb P^1}(k)$.
\item[3)] If $c_1(\mathcal E)=-1$ then $k\geq 0$ and $\mathcal E|_L=\mathcal O_{\mathbb P^1}(-k-1)\oplus \mathcal O_{\mathbb P^1}(k)$.
\end{itemize}
\end{proposition}

\begin{proof}
Let $s$ be a nonzero section in $H^0(\mathbb P^2, \mathcal E(-k))$. Using $s$ we can build a short exact sequence of sheaves $$0\rightarrow \mathcal O_{\mathbb P^2} \rightarrow \mathcal E(-k) \rightarrow \mathcal Q(-k) \rightarrow 0$$ which we can twist by $\mathcal O_{\mathbb P^2}(k)$ to get the short exact sequence of sheaves $$0\rightarrow \mathcal O_{\mathbb P^2}(k) \rightarrow \mathcal E \rightarrow \mathcal Q \rightarrow 0.$$
If $s$ is not regular (i.e. its vanishing locus is not of codimension 2 or greater), then the vanishing locus of $s$ contains a curve component. This curve is of codimension 1 in $\mathbb P^2$ thus can be identified with a form $F\in R$. If we factor out $F$ from $s$ we obtain a nonzero section $s^{\prime} \in H^0(\mathbb P^2, \mathcal E(-k-d))$ where $d$ is the degree of $F$ (see \cite{barth1977some}, Lemma 2 on page 128). Since $k$ is the largest integer such that $H^0(\mathbb P^2, \mathcal E(-k))\neq 0$, we get a contradiction. Therefore $s$ is regular.

Suppose first that $c_1(\mathcal E)=0$. If $L=\mathbb P^1$ is a general line in $\mathbb P^2$ then $L$ does not meet the zero locus of $s$. As a consequence, the restriction of the short exact sequence to $L$ is still a short exact sequence and by Chern class considerations, the restriction of $\mathcal Q$ to $L$ is $\mathcal O_{\mathbb P^1}(-k)$. Thus, restricting the exact sequence to $L$ leads to 
$$0\rightarrow \mathcal O_{\mathbb P^1}(k) \rightarrow \mathcal E|_{\mathbb P^1} \rightarrow \mathcal O_{\mathbb P^1}(-k)\rightarrow 0.$$
Since $\mathcal E$ has rank 2, is unstable, and has $c_1=0$, we know that $H^0(\mathbb P^2, \mathcal E(-1))\neq 0$ thus we can conclude that $k>0$. Using this fact,  we can conclude that $Ext^1(\mathcal O_{\mathbb P^1}(-k),\mathcal O_{\mathbb P^1}(k))=0$.
As a consequence, $\mathcal E|_{\mathbb P^1}=\mathcal O_{\mathbb P^1}(-k) \oplus \mathcal O_{\mathbb P^1}(k)$.

Now suppose that $c_1(\mathcal E)=-1$. Like before, the restriction of the short exact sequence to $L$ is still a short exact sequence except now, by Chern class considerations, the restriction of $\mathcal Q$ to $L$ is $\mathcal O_{\mathbb P^1}(-k-1)$. Thus we get the short exact sequence $$0\rightarrow \mathcal O_{\mathbb P^1}(k) \rightarrow \mathcal E|_{\mathbb P^1} \rightarrow \mathcal O_{\mathbb P^1}(-k-1)\rightarrow 0.$$ Since $\mathcal E$ has rank 2, is unstable, and has $c_1=-1$, we know that $H^0(\mathbb P^2, \mathcal E)\neq 0$ thus we can conclude that $k\geq 0$. Using this fact,  we can conclude that $Ext^1(\mathcal O_{\mathbb P^1}(-k-1),\mathcal O_{\mathbb P^1}(k))=0$.
As a consequence, $\mathcal E|_{\mathbb P^1}=\mathcal O_{\mathbb P^1}(-k-1) \oplus \mathcal O_{\mathbb P^1}(k)$.
\end{proof}

\

\begin{proposition}\label{bigprop2}
If $\mathcal E$ is an unstable, normalized, rank 2 vector bundle on $\mathbb P^2$ with index of instability $k$
 then \begin{itemize}
 \item If $c_1(\mathcal E)=0$ then $h^0(\mathbb P^2, \mathcal E(t)) = {k+t+2\choose 2}$ for $ t<k$.
 \item If $c_1(\mathcal E)=-1$ then $h^0(\mathbb P^2, \mathcal E(t)) = {k+t+2\choose 2}$ for $ t\leq k$.
 \end{itemize}
\end{proposition}

\begin{proof} Consider the exact sequence $$0\rightarrow \mathcal E(t-1) \rightarrow \mathcal E(t) \rightarrow \mathcal E(t)|_L \rightarrow 0.$$ If we apply the global section functor we get the exact sequence $$0 \rightarrow H^0(\mathbb P^2, \mathcal E(t-1)) \rightarrow H^0(\mathbb P^2, \mathcal E(t)) \rightarrow H^0(\mathbb P^2, \mathcal E(t)|_L) \rightarrow \dots$$
From this exact sequence, we have $$h^0(\mathbb P^2, \mathcal E(t)) \leq h^0(\mathbb P^2, \mathcal E(t-1)) + h^0(\mathbb P^2, \mathcal E(t)|_L).$$
If $L$ is a general line then from the previous proposition we have that 

\

\begin{tikzcd}[row sep=5pt]
{\rm if} \ c_1(\mathcal E)=0 \ {\rm then} \ h^0(\mathbb P^2,\mathcal E(t)|_L) = h^0(\mathbb P^1, \mathcal O_{\mathbb P^1}(-k+t) \oplus \mathcal O_{\mathbb P^1}(k+t))\\
{\rm if} \ c_1(\mathcal E)=-1 \ {\rm then} \ h^0(\mathbb P^2,\mathcal E(t)|_L) = h^0(\mathbb P^1, \mathcal O_{\mathbb P^1}(-k-1+t) \oplus \mathcal O_{\mathbb P^1}(k+t))
\end{tikzcd}

\

 As a consequence
 
 \
 
 \begin{tikzcd}[row sep=5pt]
 {\rm if}\ c_1(\mathcal E)=0 \ {\rm and \ if} \ t<2k \ {\rm then} \ h^0(\mathbb P^2,\mathcal E(-k+t)|_L)=max\{0,t+1\}\\
 {\rm if}\ c_1(\mathcal E)=-1 \ {\rm and \ if} \ t\leq 2k \ {\rm then} \ h^0(\mathbb P^2,\mathcal E(-k+t)|_L)=max\{0,t+1\}
 \end{tikzcd}
 
 \
 
 Since there exists a nonzero section $s \in H^0(\mathbb P^2, \mathcal E(-k))$, we can tensor this section by forms of degree $t$ and produce sections in $H^0(\mathbb P^2, \mathcal E(-k+t))$. As a consequence, we have $$h^0(\mathbb P^2, \mathcal E(t))\geq {k+t+2\choose 2}  \ \ {\rm or\ equivalently}\ \ h^0(\mathbb P^2, \mathcal E(-k+t))\geq {t+2\choose 2}$$
We can now establish the claim of the proposition by an inductive approach. In the interest of space, we let $h^0(\mathcal E)$ denote $h^0(\mathbb P^2, \mathcal E)$. Recalling that $h^0(\mathcal E(-k-1))=0$ and that $h^0(\mathbb P^2,\mathcal E(-k+t)|_L)=t+1$ (provided $t$ is in the proper range) we have the following inequalities:

\

\begin{tikzcd}[row sep=5pt]
1\leq h^0(\mathcal E(-k+0)) \leq h^0( \mathcal E(-k-1)) + h^0(\mathcal E(-k+0)|_L)=0+1=1\\
3\leq h^0(\mathcal E(-k+1)) \leq h^0(\mathcal E(-k+0)) + h^0(\mathcal E(-k+1)|_L)=1+2=3\\
6\leq h^0(\mathcal E(-k+2)) \leq h^0( \mathcal E(-k+1)) + h^0( \mathcal E(-k+2)|_L)=3+3=6\\
\dots\\
{t+2\choose 2} \leq h^0( \mathcal E(-k+t)) \leq h^0(\mathcal E(-k+t-1)) + h^0(\mathcal E(-k+t)|_L)={t+1\choose 2}+t+1={t+2\choose 2} \\
\dots
\end{tikzcd}

\

\

For $c_1(\mathcal E)=0$, following the inequalities through one at a time leads to the constraint $${t+2\choose 2} \leq h^0( \mathcal E(-k+t))\leq {t+2\choose 2} \ \ {\rm for} \ \ t<2k$$ or equivalently $${k+t+2\choose 2} \leq h^0( \mathcal E(t))\leq {k+t+2\choose 2}  \ \ {\rm for} \ \ t<k.$$ Thus we conclude that $${\rm if} \ c_1(\mathcal E)=0\ {\rm then}\  h^0(\mathbb P^2, \mathcal E(t)) = {k+t+2\choose 2}\ {\rm for} \  t<k.$$

 In a similar manner, we can also conclude that $${\rm if} \ c_1(\mathcal E)=-1\ {\rm then}\  h^0(\mathbb P^2, \mathcal E(t)) = {k+t+2\choose 2}\ {\rm for} \  t\leq k.$$
\end{proof}

\

\begin{theorem}\label{main}
Let $\mathcal E$ be a normalized, rank 2, locally free sheaf on $\mathbb P^2$. Let $L\in \mathbb K[x,y,z]$ be a general linear form. Let $H^1_*(\mathbb P^2, \mathcal E)=\oplus_{t\in \mathbb Z} H^1(\mathbb P^2, \mathcal E(t))$. Let $\phi_L(t): H^1(\mathbb P^2, \mathcal E(t-1))\rightarrow H^1(\mathbb P^2, \mathcal E(t))$ be the linear map induced by $L$.

\

1) $H^1_*(\mathbb P^2, \mathcal E)$ has the Weak Lefschetz Property

\

2) Let $\mathcal E$ be stable.

\begin{itemize}
\item If $c_1(\mathcal E)=0$ then $\phi_L(t)$ is injective for $t\leq -1$ and surjective for $t\geq -1$
\item If $c_1(\mathcal E)=-1$ then $\phi_L(t)$ is injective for $t\leq -1$ and surjective for $t\geq 0$.
\end{itemize}

\

3) Let $\mathcal E$ be unstable with index of instability $k$.

\begin{itemize}
\item If $c_1(\mathcal E)=0$ then $\phi_L(t)$ is injective for $t\leq k-1$ and  surjective for $t\geq -k-1$
\item If $c_1(\mathcal E)=-1$ then $\phi_L(t)$ is injective for $t\leq k$ and surjective for $t\geq -k-1$
\end{itemize}

\end{theorem}

\begin{proof}
In order to prove the theorem, we will first prove 2) and 3) which immediately imply 1). 

\

Consider the short exact sequence of sheaves 
\begin{equation}\label{eq0}
0\rightarrow \mathcal E(t-1) \rightarrow \mathcal E(t) \rightarrow \mathcal E(t)|_L \rightarrow 0.
\end{equation}

 If we apply the global section functor we get the long exact sequence

\

\begin{tikzcd}
 0\arrow{r} & H^0(\mathbb P^2,\mathcal E(t-1))\arrow{r} &  H^0(\mathbb P^2,\mathcal E(t))\arrow{r} & H^0(\mathbb P^2,\mathcal E(t)|_L) \\
  {  }  \arrow{r} & H^1(\mathbb P^2,\mathcal E(t-1))\arrow{r} &  H^1(\mathbb P^2,\mathcal E(t))\arrow{r} & H^1(\mathbb P^2,\mathcal E(t)|_L) \\
  {  }  \arrow{r} & H^2(\mathbb P^2,\mathcal E(t-1))\arrow{r} &  H^2(\mathbb P^2,\mathcal E(t))\arrow{r} & H^2(\mathbb P^2,\mathcal E(t)|_L)=0
\end{tikzcd}

\medskip

To show that $H^1_*(\mathbb P^2, \mathcal E)$ has the Weak Lefschetz Property, we need to show that for each $t\in \mathbb Z$, the map $H^1(\mathbb P^2,\mathcal E(t-1)) \rightarrow H^1(\mathbb P^2,\mathcal E(t))$ is either injective or surjective. From the long exact sequence above, we have the following observations:

\begin{itemize}
\item The map is injective if and only if  $h^0(\mathbb P^2, \mathcal E(t-1)) - h^0(\mathbb P^2, \mathcal E(t)) + h^0(\mathbb P^2, \mathcal E(t)|_L)=0.$
\item The map is injective if $h^0(\mathbb P^2, \mathcal E(t)|_L)=0.$
\item The map is surjective if and only if  $h^1(\mathbb P^2, \mathcal E(t)|_L) - h^2(\mathbb P^2, \mathcal E(t-1)) + h^2(\mathbb P^2, \mathcal E(t))=0.$
\item The map is surjective if $h^1(\mathbb P^2, \mathcal E(t)|_L)=0.$
\end{itemize}

If the generic splitting type of $\mathcal E$  is $\mathcal O_{\mathbb P^1}(a) \oplus \mathcal O_{\mathbb P^1}(b)$ then, by Serre Duality, $h^1(\mathbb P^2,\mathcal E|_L) = h^0(\mathbb P^1, \mathcal O_{\mathbb P^1}(-a-2) \oplus \mathcal O_{\mathbb P^1}(-b-2))$ and $h^1(\mathbb P^2,\mathcal E(t)|_L) = h^0(\mathbb P^1, \mathcal O_{\mathbb P^1}(-a-t-2) \oplus \mathcal O_{\mathbb P^1}(-b-t-2))$. As a consequence, we can easily compute the value of $h^1(\mathbb P^2,\mathcal E(t)|_L)$. In particular, if $-a-t-2\leq -1$ and $-b-t-2\leq -1$ then $h^1(\mathbb P^2,\mathcal E(t)|_L)=0$. We collect the following facts:

\begin{itemize}
\item[A)] Since $\mathcal E$ is locally free on $\mathbb P^2$, by duality we have $h^2(\mathbb P^2, \mathcal E(t)) = h^0(\mathbb P^2, \mathcal E^\vee(-t-3))$. 
\item[B)] If we restrict $\mathcal E$ to a general line $L$ we have $h^1(\mathbb P^2,\mathcal E(t)|_L)=h^0(\mathbb P^2,\mathcal E^\vee(-t-2)|_L)$.
\item[C)] Since $\mathcal E$ has rank two, if $c_1(\mathcal E)=0$ then $\mathcal E^\vee \cong \mathcal E$ and if $c_1(\mathcal E)=-1$ then $\mathcal E^\vee \cong \mathcal E(1)$. 
\end{itemize}

We now assume that $\mathcal E$ is stable and use the above considerations to establish a range of values of $t$ where the map, $\phi_L(t): H^1(\mathbb P^2,\mathcal E(t-1)) \rightarrow H^1(\mathbb P^2,\mathcal E(t))$, is injective and a range of values where the map is surjective. It is important to note that the following shows that for {\it every} value of $t$, the map is either injective or surjective.

\

Suppose $\mathcal E$ is stable and that $c_1(\mathcal E)=0$. By Proposition~\ref{GM}, $\mathcal E$ splits on $L$ as $\mathcal O_{\mathbb P^1}\oplus \mathcal O_{\mathbb P^1}$. In this case, $h^0(\mathbb P^2, \mathcal E(t)|_L)=0$ for $t\leq -1$ and $h^1(\mathbb P^2, \mathcal E(t)|_L)=0$ for $t\geq -1$. Thus $\phi_L(t)$ is injective for $t\leq -1$ and surjective for $t\geq -1$.

\

 Suppose $\mathcal E$ is stable and that $c_1(\mathcal E)=-1$. By Proposition~\ref{GM}, $\mathcal E$ splits on $L$ as $\mathcal O_{\mathbb P^1}(-1)\oplus \mathcal O_{\mathbb P^1}$. In this case, $h^0(\mathbb P^2, \mathcal E(t)|_L)=0$ for $t\leq -1$ and $h^1(\mathbb P^2, \mathcal E(t)|_L)=0$ for $t\geq 0$.

\

 Suppose $\mathcal E$ is unstable and that $c_1(\mathcal E)=0$. If the index of instability is $k$ then by Proposition~\ref{bigprop}, $k>0$ and $\mathcal E|_L=\mathcal O_{\mathbb P^1}(-k)\oplus \mathcal O_{\mathbb P^1}(k)$. In this case, Proposition~\ref{bigprop2} allows us to conclude that $$h^0(\mathbb P^2, \mathcal E(t-1)) - h^0(\mathbb P^2, \mathcal E(t)) + h^0(\mathbb P^2, \mathcal E(t)|_L)=0 \ \ {\rm for} \ \ t\leq k-1.$$ This implies that $\phi_L(t)$ is injective for $t\leq k-1$. Using A) and B) above, we note that $$h^1(\mathbb P^2, \mathcal E(t)|_L) - h^2(\mathbb P^2, \mathcal E(t-1)) + h^2(\mathbb P^2, \mathcal E(t))$$ can be expressed as $$h^0(\mathbb P^2, \mathcal E^\vee(-t-2)|_L) - h^0(\mathbb P^2, \mathcal E^\vee(-t-2))+ h^0(\mathbb P^2, \mathcal E^\vee(-t-3)).$$ Using C) and rearranging, we can then express this as $$h^0(\mathbb P^2, \mathcal E(-t-3) - h^0(\mathbb P^2, \mathcal E(-t-2))+ h^0(\mathbb P^2, \mathcal E(-t-2)|_L).$$ By Proposition~\ref{bigprop2} this quantity is equal to $0$ for $-t-2\leq k-1$. In other words, $\phi_L(t)$ is surjective for $-k-1\leq t$.   

\

Suppose $\mathcal E$ is unstable and that $c_1(\mathcal E)=-1$. If the index of instability is $k$ then by Proposition~\ref{bigprop}, $k\geq 0$ and $\mathcal E|_L=\mathcal O_{\mathbb P^1}(-k-1)\oplus \mathcal O_{\mathbb P^1}(k)$. In this case, Proposition~\ref{bigprop2} allows us to conclude that  $$h^0(\mathbb P^2, \mathcal E(t-1)) - h^0(\mathbb P^2, \mathcal E(t)) + h^0(\mathbb P^2, \mathcal E(t)|_L)=0 \ \ {\rm for} \ \  t\leq k.$$ This implies that $\phi_L(t)$ is injective for $t\leq k$. Using A) and B) above, we note that $$h^1(\mathbb P^2, \mathcal E(t)|_L) - h^2(\mathbb P^2, \mathcal E(t-1)) + h^2(\mathbb P^2, \mathcal E(t))$$ can be expressed as $$h^0(\mathbb P^2, \mathcal E^\vee(-t-2)|_L) - h^0(\mathbb P^2, \mathcal E^\vee(-t-2))+ h^0(\mathbb P^2, \mathcal E^\vee(-t-3)).$$ Using C) and rearranging, we can then express this as $$h^0(\mathbb P^2, \mathcal E(-t-2) - h^0(\mathbb P^2, \mathcal E(-t-1))+ h^0(\mathbb P^2, \mathcal E(-t-1)|_L).$$ By Proposition~\ref{bigprop2} this quantity is equal to $0$ for $-t-1\leq k$. In other words, $\phi_L(t)$ is surjective for $-k-1\leq t$.   

\

In each of these cases, we see that for each $t\in \mathbb Z$, the map $H^1(\mathbb P^2,\mathcal E(t-1)) \rightarrow H^1(\mathbb P^2,\mathcal E(t))$ is either injective or surjective. Thus $H^1_*(\mathbb P^2, \mathcal E)$ has the Weak Lefschetz Property for any rank 2 vector bundle $\mathcal E$ on $\mathbb P^2$.
\end{proof}

\

\begin{corollary} If $F_1,F_2,F_3$ is a regular sequence in $R=\mathbb K[x,y,z]$ then $R/(F_1,F_2,F_3)$ has the Weak Lefschetz Property.
\end{corollary}

\begin{corollary}  If $\mathcal E$ is a rank 2 vector bundle on $\mathbb P^2$ then $H^1_*(\mathbb P^2, \mathcal E)$ is unimodal.
\end{corollary}

\begin{proof}  In the proof of Theorem~\ref{main}, we saw that for any rank 2 vector bundle $\mathcal E$ on $\mathbb P^2$, there exists an $r$ such that for $t<r$ the map $\times L: H^1(\mathbb P^2,\mathcal E(t-1)) \rightarrow H^1(\mathbb P^2,\mathcal E(t))$ is injective and for $t\geq r$ the map $\times L: H^1(\mathbb P^2,\mathcal E(t-1)) \rightarrow H^1(\mathbb P^2,\mathcal E(t))$ was surjective. This fact establishes unimodality.
\end{proof}

\section{An Example and Further Remarks}
In this section, we first give an example to illustrate the theorems of the paper and the structure of the Buchsbaum-Rim complexes. In each of the following two examples, the associated locally free sheaf is unstable. After giving the two examples, we conclude the paper with a few remarks and considerations for possible further research.

\begin{example} Consider a map $\phi: R(-7)\oplus R(-2)^3 \rightarrow R(-1)\oplus R$ whose cokernel is a finite length module $\mathbf M$. An elementary computation show that $\mathbf M=\mathbf M_0 \oplus \dots \oplus \mathbf M_9$ has Hilbert function $(1,4,6,7,7,7,7,6,4,1)$. The Buchsbaum-Rim complex associated to $\phi$ is:
\begin{equation}\label{exam}
0\rightarrow R(-12)\oplus R(-11) \rightarrow R(-10)^3 \oplus R(-5) \rightarrow R(-7)\oplus R(-2)^3 \rightarrow R(-1)\oplus R \rightarrow \mathbf M \rightarrow 0
\end{equation}

If we sheafify (\ref{exam}) and tensor by $\mathcal O_{\mathbb P^2}(6)$ we get the exact sequence
\begin{equation}\label{exam2}
0\rightarrow \mathcal O_{\mathbb P^2}(-6)\oplus \mathcal O_{\mathbb P^2}(-5) \rightarrow \mathcal O_{\mathbb P^2}(-4)^3 \oplus \mathcal O_{\mathbb P^2}(1) \rightarrow \mathcal O_{\mathbb P^2}(-1)\oplus \mathcal O_{\mathbb P^2}(4)^3 \rightarrow \mathcal O_{\mathbb P^2}(5)\oplus \mathcal O_{\mathbb P^2}(6) \rightarrow 0
\end{equation}

We can break (\ref{exam2}) into two short exact sequences
\begin{equation}\label{exam3}
0\rightarrow \mathcal O_{\mathbb P^2}(-6)\oplus \mathcal O_{\mathbb P^2}(-5) \rightarrow \mathcal O_{\mathbb P^2}(-4)^3 \oplus \mathcal O_{\mathbb P^2}(1) \rightarrow \mathcal E\rightarrow 0
\end{equation}

and

\begin{equation}\label{exam4}
0\rightarrow \mathcal E\rightarrow \mathcal O_{\mathbb P^2}(-1)\oplus \mathcal O_{\mathbb P^2}(4)^3 \rightarrow \mathcal O_{\mathbb P^2}(5)\oplus \mathcal O_{\mathbb P^2}(6) \rightarrow 0
\end{equation}
where $\mathcal E$ is a normalized rank 2 locally free sheaf with $c_1(\mathcal E)=0$. From the exact sequence (\ref{exam3}), we see that $H^0(\mathbb P^2, \mathcal E(-1))>0$ and $H^0(\mathbb P^2, \mathcal E(-2))=0$. Therefore $\mathcal E$ is unstable with index of instability $k=1$. By Theorem~\ref{main}, $\phi_L(t)$ is injective for $t\leq 0$ and  surjective for $t\geq -2$
This corresponds to saying that the map $\times L:\mathbf M_{d-1}\rightarrow \mathbf M_d$ is injective for $d\leq 6$ and surjective for $d\geq 4$. Note that this implies bijectivity for $4\leq d\leq 6$ thus $\mathbf M_3, \mathbf M_4, \mathbf M_5$ and $\mathbf M_6$ all have the same dimension. Further note that for every value of $d$, the map $\times L:\mathbf M_{d-1}\rightarrow \mathbf M_d$ is either injective or surjective thus $\mathbf M$ has the Weak Lefschetz Property. 
\end{example}

\begin{example} Consider a map $\phi: R(-8) \oplus R(-2)^4 \rightarrow R(-1)\oplus R^2$ whose cokernel is a finite length module $\mathbf M$. An elementary computation show that $\mathbf M=\mathbf M_0 \oplus \dots \oplus \mathbf M_{12}$ has Hilbert function $(2,7,11,14,16,17,17,17,16,14,11,7,2)$. The Buchsbaum-Rim complex associated to $\phi$ is:
\begin{equation}\label{exam5}
0\rightarrow R(-15)^2\oplus R(-14) \rightarrow R(-13)^4 \oplus R(-7) \rightarrow R(-8)\oplus R(-2)^4 \rightarrow R(-1)\oplus R^2 \rightarrow \mathbf M \rightarrow 0
\end{equation}

If we sheafify (\ref{exam5}) and tensor by $\mathcal O_{\mathbb P^2}(7)$ we get the exact sequence
\begin{equation}\label{exam6}
0\rightarrow \mathcal O_{\mathbb P^2}(-8)^2\oplus \mathcal O_{\mathbb P^2}(-7) \rightarrow \mathcal O_{\mathbb P^2}(-6)^4 \oplus \mathcal O_{\mathbb P^2} \rightarrow \mathcal O_{\mathbb P^2}(-1)\oplus \mathcal O_{\mathbb P^2}(5)^4 \rightarrow \mathcal O_{\mathbb P^2}(6)\oplus \mathcal O_{\mathbb P^2}(7)^2\rightarrow 0
\end{equation}

We can break (\ref{exam6}) into two short exact sequences
\begin{equation}\label{exam7}
0\rightarrow \mathcal O_{\mathbb P^2}(-8)^2\oplus \mathcal O_{\mathbb P^2}(-7) \rightarrow \mathcal O_{\mathbb P^2}(-6)^4 \oplus \mathcal O_{\mathbb P^2} \rightarrow \mathcal E\rightarrow 0
\end{equation}

and

\begin{equation}\label{exam8}
0\rightarrow \mathcal E\rightarrow \mathcal O_{\mathbb P^2}(-1)\oplus \mathcal O_{\mathbb P^2}(5)^4 \rightarrow \mathcal O_{\mathbb P^2}(6)\oplus \mathcal O_{\mathbb P^2}(7)^2 \rightarrow 0
\end{equation}
where $\mathcal E$ is a normalized rank 2 locally free sheaf with $c_1(\mathcal E)=-1$. From exact sequence (\ref{exam7}), we see that $H^0(\mathbb P^2, \mathcal E)>0$ and $H^0(\mathbb P^2, \mathcal E(-1))=0$. Therefore $\mathcal E$ is unstable with index of instability $k=0$. By Theorem~\ref{main}, $\phi_L(t)$ is injective for $t\leq 0$ and  surjective for $t\geq -1$
This corresponds to saying that the map $\times L:\mathbf M_{d-1}\rightarrow \mathbf M_d$ is injective for $d\leq 7$ and surjective for $d\geq 6$. Note that this implies bijectivity for $6\leq d\leq 7$ thus $\mathbf M_5, \mathbf M_6, \mathbf M_7$ all have the same dimension. Further note that for every value of $d$, the map $\times L:\mathbf M_{d-1}\rightarrow \mathbf M_d$ is either injective or surjective thus $\mathbf M$ has the Weak Lefschetz Property. 
\end{example}

In this paper, we have shown that the first cohomology module of a rank two locally free sheaf on $\mathbb P^2$ must have the Weak Lefschetz Property. This is equivalent to showing that if $\mathbf M$ is a finite length module arising as the cokernel of a map of the form $\phi:\mathbf F \rightarrow \mathbf G$ with $\mathbf F=\oplus_{i=1}^{n+2} R(-a_i)$ and $\mathbf G=\oplus_{i=1}^{n} R(-b_i)$, then $\mathbf M$ has the Weak Lefschetz Property. As a special case, every height three Artinian complete intersection has the Weak Lefschetz Property (proved first in \cite{harima2003weak} and proved again in \cite{brenner2007syzygy}). 

The key piece needed in the proofs of the main theorems is that $\mathcal E$ is a rank two locally free sheaf on a surface. Many of the key conclusions ultimately follow from this fact. This suggests that there may be generalizations of Theorem~\ref{main} to the case of rank two locally free sheaves on weighted projective planes and on $\mathbb P^1 \times \mathbb P^1$. We note the interesting paper by Harima and Watanabe where they considered the strong Lefschetz property for Artinian algebras with non-standard grading \cite{harima2007strong}. It is hoped that additional progress may be made in the understanding of Lefschetz Properties by considering the more general problem for modules over non-standard graded rings.

\

{\bf Acknowledgments}: The third author wishes to express thanks to a brief but helpful conversation with Aaron Bertram which helped crystallize some ideas in the proof of Proposition~\ref{bigprop}.
This paper is based on research partially supported by a grant of the group GNSAGA of INdAM and by the National Science Foundation under grant  \#1322508, \#1633830, and \#1712788.

\bibliographystyle{plain}
\bibliography{WLP}

\begin{thebibliography}{10}

\bibitem{barth1977some}
Wolf Barth.
\newblock Some properties of stable rank-2 vector bundles on $\mathbb {P}^n$.
\newblock {\em Mathematische Annalen}, 226(2):125--150, 1977.

\bibitem{boij1999components}
Mats Boij.
\newblock Components of the space parametrizing graded {G}orenstein {A}rtin
  algebras with a given {H}ilbert function.
\newblock {\em Pacific Journal of Mathematics}, 187(1):1--11, 1999.

\bibitem{brenner2007syzygy}
Holger Brenner and Almar Kaid.
\newblock Syzygy bundles on $\mathbb {P}^2$ and the weak {L}efschetz property.
\newblock {\em Illinois Journal of Mathematics}, 51(4):1299--1308, 2007.

\bibitem{buchsbaum1964generalized}
David~A Buchsbaum.
\newblock A generalized {K}oszul complex {I}.
\newblock {\em Transactions of the American Mathematical Society},
  111(2):183--196, 1964.

\bibitem{buchsbaum1964generalizedb}
David~A Buchsbaum and Dock~S Rim.
\newblock A generalized {K}oszul complex {II}. {D}epth and multiplicity.
\newblock {\em Transactions of the American Mathematical Society},
  111(2):197--224, 1964.

\bibitem{cook2011weak}
David Cook~II and Uwe Nagel.
\newblock The weak {L}efschetz property, monomial ideals, and lozenges.
\newblock {\em Illinois Journal of Mathematics}, 55(1):377--395, 2011.

\bibitem{eisenbud2013commutative}
David Eisenbud.
\newblock {\em Commutative Algebra: with a view toward algebraic geometry},
  volume 150.
\newblock Springer Science \& Business Media, 2013.

\bibitem{grauert1975vektorbundel}
Hans Grauert and Gerhard M{\"u}lich.
\newblock Vektorb{\"u}ndel vom {R}ang 2 {\"u}ber dem n-dimensionalen
  komplex-projektiven {R}aum.
\newblock {\em manuscripta mathematica}, 16(1):75--100, 1975.

\bibitem{harima1995characterization}
Tadahito Harima.
\newblock Characterization of {H}ilbert functions of {G}orenstein {A}rtin
  algebras with the weak {S}tanley property.
\newblock {\em Proceedings of the American Mathematical Society},
  123(12):3631--3638, 1995.

\bibitem{harima2013lefschetz}
Tadahito Harima, Toshiaki Maeno, Hideaki Morita, Yasuhide Numata, Akihito
  Wachi, and Junzo Watanabe.
\newblock Lefschetz properties.
\newblock In {\em The Lefschetz Properties}, pages 97--140. Springer, 2013.

\bibitem{harima2003weak}
Tadahito Harima, Juan~C Migliore, Uwe Nagel, and Junzo Watanabe.
\newblock The weak and strong {L}efschetz properties for {A}rtinian
  {K}-algebras.
\newblock {\em Journal of Algebra}, 262(1):99--126, 2003.

\bibitem{harima2007strong}
Tadahito Harima and Junzo Watanabe.
\newblock The strong {L}efschetz property for {A}rtinian algebras with
  non-standard grading.
\newblock {\em Journal of Algebra}, 311(2):511--537, 2007.

\bibitem{hartshorne2013algebraic}
Robin Hartshorne.
\newblock {\em Algebraic geometry}, volume~52.
\newblock Springer Verlag, New York, 1977.

\bibitem{iarrobino1994associated}
Anthony~Ayers Iarrobino.
\newblock Associated graded algebra of a {G}orenstein {A}rtin algebra.
\newblock {\em Memoirs of the AMS}, 514, 1994.

\bibitem{kreuzer2000determinantal}
Martin Kreuzer, Juan~C Migliore, Chris Peterson, and Uwe Nagel.
\newblock Determinantal schemes and {B}uchsbaum--{R}im sheaves.
\newblock {\em Journal of Pure and Applied Algebra}, 150(2):155--174, 2000.

\bibitem{li2010monomial}
Jizhou Li and Fabrizio Zanello.
\newblock Monomial complete intersections, the weak {L}efschetz property and
  plane partitions.
\newblock {\em Discrete Mathematics}, 310(24):3558--3570, 2010.

\bibitem{mezzetti2013laplace}
Emilia Mezzetti, Rosa~M Mir{\'o}-Roig, and Giorgio Ottaviani.
\newblock Laplace equations and the weak {L}efschetz property.
\newblock {\em Canadian Journal of Mathematics}, 65(3):634--654, 2013.

\bibitem{migliore2011monomial}
Juan Migliore, Rosa Mir{\'o}-Roig, and Uwe Nagel.
\newblock Monomial ideals, almost complete intersections and the weak
  {L}efschetz property.
\newblock {\em Transactions of the American Mathematical Society},
  363(1):229--257, 2011.

\bibitem{migliore2013survey}
Juan Migliore and Uwe Nagel.
\newblock Survey article: a tour of the weak and strong {L}efschetz properties.
\newblock {\em Journal of Commutative Algebra}, 5(3):329--358, 2013.

\bibitem{migliore1999buchsbaum}
Juan~C Migliore, Uwe Nagel, and Chris Peterson.
\newblock Buchsbaum--{R}im sheaves and their multiple sections.
\newblock {\em Journal of Algebra}, 219(1):378--420, 1999.

\bibitem{okonek1980vector}
Christian Okonek, Michael Schneider, and Heinz Spindler.
\newblock {\em Vector bundles on complex projective spaces}, volume~3.
\newblock Springer, 1980.

\bibitem{stanley1980number}
Richard~P Stanley.
\newblock The number of faces of a simplicial convex polytope.
\newblock {\em Advances in Mathematics}, 35(3):236--238, 1980.

\bibitem{stanley1980weyl}
Richard~P Stanley.
\newblock Weyl groups, the hard {L}efschetz theorem, and the {S}perner
  property.
\newblock {\em SIAM Journal on Algebraic Discrete Methods}, 1(2):168--184,
  1980.

\bibitem{watanabe1998note}
Junzo Watanabe.
\newblock A note on complete intersections of height three.
\newblock {\em Proceedings of the American Mathematical Society},
  126(11):3161--3168, 1998.

\bibitem{zanello2006non}
Fabrizio Zanello.
\newblock A non-unimodal codimension 3 level h-vector.
\newblock {\em Journal of Algebra}, 305(2):949--956, 2006.

\end{thebibliography}

\end{document}